\documentclass{article}

\usepackage[
	a4paper,
	margin=1in
]{geometry}
\usepackage{setspace}
\usepackage{sectsty} 
\usepackage{titlesec}
\usepackage{appendix}

\usepackage{amsmath, amssymb, amsthm}
\usepackage{mathtools}

\usepackage[
	setpagesize=false,
	colorlinks=true,
	linkcolor=BrickRed,
        citecolor=OliveGreen,
        urlcolor=black,
	pdfencoding=auto,
	psdextra,
]{hyperref}
\usepackage{cleveref}

\usepackage{etoolbox} 
\usepackage[shortlabels]{enumitem}
\usepackage{xparse} 

\usepackage[dvipsnames]{xcolor}
\usepackage{graphicx, color}
\usepackage{epsfig}
\usepackage{tikz}
\usetikzlibrary{calc,decorations.pathmorphing,decorations.text}
\usepackage{subcaption}
\usepackage{refcount}

\usepackage{todonotes}
\usepackage[
    deletedmarkup=sout,
    commentmarkup=uwave,
    authormarkuptext=name
]{changes}

\usepackage{comment}

\setlist[enumerate,1]{label=(\arabic*), ref=(\arabic*)}
\setlist[enumerate,3]{label=(\roman*), ref=(\roman*)}

\allsectionsfont{\boldmath} 

\titlelabel{\thetitle.\quad}

\theoremstyle{plain}
\newtheorem{theorem}{Theorem}[section]
\newtheorem{lemma}[theorem]{Lemma}

\newtheorem{question}[theorem]{Question}

\newtheorem{claim}[theorem]{Claim}
\newtheorem*{claim*}{Claim}

\makeatletter
\newenvironment{claimproof}[1][Proof]{\par
	\pushQED{\qed}%
	
	\normalfont \topsep6\p@\@plus6\p@\relax
	\trivlist
	\item[\hskip\labelsep
	\textit{#1}\@addpunct{.}~]\ignorespaces
}{%
	\popQED\endtrivlist\@endpefalse
}
\makeatother

\newlist{Cases}{enumerate}{3}
\setlist[Cases]{parsep=0pt plus 1pt}
\setlist[Cases,1]{wide=0pt, listparindent=\parindent,
    label = \textbf{Case~\arabic*:}, ref = \arabic*}
\setlist[Cases,2]{wide=\parindent, listparindent=\parindent,
    label = \textbf{Case~\arabic{Casesi}-\arabic{Casesii}:}}

\crefname{Casesi}{case}{cases}

\newcounter{case}
\AtBeginEnvironment{proof}{\setcounter{case}{0}}

\crefname{case}{case}{cases}

\theoremstyle{definition}
\newtheorem{definition}[theorem]{Definition}

\newcommand{\PP}{\mathbb{P}}
\newcommand{\EE}{\mathbb{E}}

\newcommand{\Var}{\mathrm{Var}}
\newcommand{\Cov}{\mathrm{Cov}}
\newcommand{\disc}{\mathrm{disc}}

\usepackage[
    backend=biber,
    sorting=nyt,
    maxbibnames=99
]{biblatex}
\title{On high discrepancy $1$-factorizations of complete graphs}
\author{Jiangdong Ai\thanks{School of Mathematical Sciences and LPMC, Nankai University, Tianjin 300071, P.R.
China. Email: {\tt jd@nankai.edu.cn., hefankang@mail.nankai.edu.cn.}} \and Fankang He\footnotemark[1] \and Seonghyuk Im \thanks{Department of Mathematical Sciences, KAIST, and Extremal Combinatorics and Probability Group (ECOPRO), Institute for Basic Science (IBS), Daejeon, South Korea. Email:{\tt $\{$seonghyuk, hyunwoo.lee$\}$@kaist.ac.kr}} \and Hyunwoo Lee\footnotemark[2]}
\date{\today}

\begin{document}

\maketitle


\begin{abstract}
    We proved that for every sufficiently large $n$, the complete graph $K_{2n}$ with an arbitrary edge signing $\sigma: E(K_{2n}) \to \{-1, +1\}$ admits a high discrepancy $1$-factor decomposition. That is, there exists a universal constant $c > 0$ such that every edge-signed $K_{2n}$ has a perfect matching decomposition $\{\psi_1, \ldots, \psi_{2n-1}\}$, where for each perfect matching $\psi_i$, the discrepancy $\lvert \frac{1}{n} \sum_{e\in E(\psi_i)} \sigma(e) \rvert$ is at least $c$. 
\end{abstract}


\section{Introduction}\label{sec:intro}

The study of finding a structure in a given system that deviates from the typical structures of the system has appeared in many contexts in various areas of mathematics over the past century. The systematic investigation of such problems is now referred to as discrepancy theory, which was initiated from the seminal work of Weyl~\cite{Weyl}. Subsequently, discrepancy theory has been applied to a variety of problems in number theory, discrete geometry, ergodic theory, and combinatorics. For a comprehensive overview of the development and applications of discrepancy theory, we suggest referring to ~\cite{Alexander-Beck-Chen,Beck-Chen}.

The concept of discrepancy in edge-colored graphs was implicitly studied in various problems of graph theory, such as many problems in Ramsey theory~\cite{Ramsey-survey,Size-ramsey-early} or several conjectures on common graphs~\cite{Erdos-common,Thomason-common} and it was explicitly mentioned in several early works~\cite{Erdos-Furedi-Loebl-Sos, Erdos-Spencer}.
To provide an intuitive framework for defining the notion of \emph{discrepancy} in graphs, we confine our consideration to a simplified edge-coloring paradigm. Specifically, we assign the colors in the edge-coloring of the graphs to be $-1$ or $+1$. 

\begin{definition}
    Let $G$ be a graph with an edge coloring $\sigma: E(G) \to \{-1, +1\}$. The  \emph{(normalized-) discrepancy} of $G$ is defined as
    $$
        \disc(G) := \Bigl\lvert \frac{1}{|E(G)|}\sum_{e\in E(G)} \sigma(e) \Bigr\rvert.
    $$
\end{definition}

Recently, the study of high discrepancy spanning structures in edge-colored graphs has attracted significant attention, leading to extensive investigations. For instance, researchers have explored Dirac-type conditions for the existence of Hamilton cycles and perfect matchings with high discrepancy in graphs~\cite{Discrepancy-Hamcycle}, directed graphs~\cite{Oriented-discrepancy}, and hypergraphs~\cite{Gishboliner-Glock-Sgueglia,Han-Lang-Marciano-PavezSigne, Lu-Ma-Xie}. Similarly, the discrepancy version of Hajnal-Szemer\'{e}di theorem~\cite{Hajnal-Szemredi-discrepancy} and the minimum degree threshold for $H$-factors with high discrepancy~\cite{H-factor-discrepancy} were considered.

In this paper, we consider a high discrepancy edge decomposition problem. As we mentioned in the previous paragraph, the existence of large spanning structures in edge-colored graphs has received a lot of attention so far. Hence, it is natural to raise the question of whether it is possible to find many edge-disjoint structures simultaneously, where each of them has a high discrepancy. In other words, can we decompose a given edge-colored graph into edge decomposition into subgraphs such that each subgraph in the decomposition graph has a high discrepancy? Our main result is the existence of a high discrepancy version of the $1$-factor decompositions in complete graphs.

\begin{theorem}\label{thm:high_discrepancy}
    There exists a universal constant $c>0$ such that the following holds for every sufficiently large $n$.
    For every edge coloring $\sigma:E(K_{2n}) \to \{-1, +1\}$ of the complete graph on $2n$ vertices has a $1$-factor decomposition $\{\psi_1, \ldots, \psi_{2n-1}\}$ of $K_{2n}$ such that for each $i \in [2n-1]$, we have $\disc(\psi_i) > c$.
\end{theorem}

    
    In the course of proving \Cref{thm:high_discrepancy}, we derive the theorem for a `low-discrepancy decomposition', which we deem to be of independent significance.
    For a graph $G$ with edge coloring $\sigma:E(G) \to \{-1, +1\}$, the \emph{signed discrepancy} of $G$ defined as 
    $$\disc^{\pm}(G) :=\frac{1}{|E(G)|} \sum_{e \in E(G)} \sigma(e).$$
    We note that $\disc(G)$ is same as $|\disc^{\pm}(G)|$.
\begin{theorem}\label{thm:balanced_coloring}
    Let $\Delta > 0$ be a positive integer and let $\varepsilon > 0$ be a positive real number. Then there exists $n_0 = n_0(\Delta, \varepsilon) > 0$ such that the following holds for all $n \geq n_0$.

    Assume $K_n$ can be edge-decomposed into $\{F_1, \dots, F_m\}$, where $\Delta(F_i) \leq \Delta$ and $e(F_i) \geq \varepsilon n$ for each $i\in [m]$.
    Then for every edge coloring $\sigma:E(K_{n}) \to \{-1, +1\}$, there exists an edge decomposition of $K_{n}$ into $\{F_1, \dots, F_m\}$ such that for each $i\in [m]$, we have $\disc^{\pm}(K_n) - \varepsilon \leq \disc^{\pm}(F_i) \leq \disc^{\pm}(K_n) + \varepsilon$.
\end{theorem}

To prove our main theorems, we employed randomness in several ways. In particular, in the proof of \Cref{thm:high_discrepancy}, we make use of Talagrand's inequality for random permutations (\Cref{lem:Talagrand}) and also adapt the dependent random choice technique. We believe our proof strategy has potential for broader applications, particularly in addressing discrepancy versions of various other problems.

\section{Preliminaries}\label{sec:prelim}
\subsection{Notations}\label{subsec:notations}
Given a positive integer $n$, let $[n]:=\{1,2, \ldots, n\}$. 
Let $G$ be a graph with a signing function $\sigma: E(G) \to \{-1,+1\}$. An edge $e$ of $G$ is called a \emph{positive edge (negative edge, resp.)} if $\sigma(e) = +1$ ($\sigma(e) = -1$, resp.). 
For a subgraph $H \subseteq G$, define $S_\sigma(H) := \sum_{e \in E(H)} \sigma(e)$.  
For a vertex $v \in V(G)$, define $N^+(v):=\{u\in V(G): uv \in E(G), \sigma(uv) = +1\}$ and $d^+(v) = |N^+(v)|$. Define $N^-(u)$ and $d^-(u)$ similarly.
Let $S_n$ denote the symmetric group of degree $n$.
For real numbers $a, b$ and $c>0$ denote $a = b \pm c$ if $a \in [b-c, b+c]$.
We omit floor and ceiling signs if these are not crucial and make no attempt to optimize the absolute constants occurring in the statements.

\subsection{Concentration inequalities}\label{subsec:inequalities}

We need to collect several concentration inequalities. 

\begin{lemma}[Chebyshev's inequality]\label{lem:Chebyshev}
    Let $X$ be a random variable with mean $\mu$ and variance $\sigma^2$, then for any $\lambda>0$, 
    \[
    \PP(|X-\mu| \geq \lambda) \leq \frac{\sigma^2}{\lambda^2}.
    \]
\end{lemma}

The following concentration inequality on random permutations is a simplified version of a theorem proved by McDiarmid~\cite[Theorem 4.1]{McDiarmid-Talagrand}.

\begin{lemma}[Talagrand's inequality on random permutations~\cite{McDiarmid-Talagrand}]\label{lem:Talagrand}
    Let $h:S_n \to \mathbb{R}_{\geq 0}$ be a function, and let $\pi \in S_n$ be a random permutation. 
    Suppose that there is a constant $c, r>0$ such that the following two conditions hold.
    \begin{itemize}
        \item For any permutations $\pi$ and $\pi'$ differing at $d$ coordinates, then we have $|h(\pi)-h(\pi')| \leq cd$.
        \item If $h(\pi) = s$, then we can find a set $S$ with at most $rs$ coordinates such that for any permutation $\pi' \in S_n$ agreeing with $\pi$ on $S$, we have $h(\pi') \geq s$.
    \end{itemize} 
     Let $M$ be the median of $h(\pi)$. Then for every $t>0$, we have $\PP(|h(\pi)-M| \geq t) \leq 6\exp(-\frac{t^2}{16rc^2(M+t)})$. 
\end{lemma}

\subsection{Decompositions of complete graphs}
We also require the following lemma, which can be deduced from the strong form of resolution of the Oberwolfach problem~{\cite[Theorem 1.3]{Oberwolfach}}. 
\begin{lemma}\label{Lem:C4-decomp}
    The following holds for sufficiently large $n$.
    \begin{itemize}
        \item If $n$ is even, then $K_{2n}$ can be decomposed into $(n-2)$ copies of $C_4$-factors and one copy of $K_4$-factor.
        \item If $n$ is odd, then $K_{2n}$ can be decomposed into $(n-2)$ copies of $\frac{n-3}{2} C_4 \cup C_6$ and one copy of $\frac{n-3}{2}K_4 \cup K_{3, 3}$.
    \end{itemize}
\end{lemma}

\section{Proof of main results}

\subsection{Proof of Theorem~\ref{thm:balanced_coloring}}
In this section, we prove \Cref{thm:balanced_coloring}. However, for later usage, we will provide an extension of it. First, we introduce some relevant terminology.
For a finite set $X$, let $S_X$ be a set of all permutations on $X$ and let $X^{(k)}$ be the set of $k$-tuples $(x_1, \ldots, x_k)$ of elements of $X$ such that all $x_i$ are distinct. In this paper, all tuples are ordered. We denote $n^{(k)} := n(n-1)\cdots(n-k+1)$, where $n \geq k \geq 1$ are positive integers.
For a permutation $\pi \in S_X$ and a family of $k$-tuples $\mathcal{F} \subseteq X^{(k)}$, we define $\pi(\mathcal{F}) = \{(\pi(x_1), \ldots, \pi(x_k)) \mid (x_1, \ldots, x_k) \in \mathcal{F}\}$.

We are now ready to state our key lemma as follows.
\begin{lemma}\label{lem:balanced_coloring}
    For every $\Delta, \varepsilon, k$, there exists $n_0=n_0(\Delta, \varepsilon, k) > 0$ and $c = c(\Delta, \varepsilon, k)>0$ such that the following holds for all $n \geq n_0$ and $p \in [0, 1]$.
    Let $X$ be an $n$-element set and $\mathcal{F} \subseteq X^{(k)}$ be a set of tuples such that $|\mathcal{F}| \geq \varepsilon n$ and let every $x \in X$ be contained in at most $\Delta$ elements of $\mathcal{F}$.
    Let $\sigma:X^{(k)} \to \{-1, +1\}$ be a signing with $|\sigma^{-1}(+1)|=pn^{(k)}$. 
    Let $\pi \in S_X$ be a permutation chosen uniformly at random. 
    Then with probability at least $1-\exp(-cn)$, we have $|\sigma^{-1}(+1) \cap \pi(\mathcal{F})| = (p \pm \varepsilon)|\mathcal{F}|$.
\end{lemma}

\begin{proof}[Proof of \Cref{lem:balanced_coloring}]
        Let $h(\pi) = |\sigma^{-1}(+1) \cap \pi(\mathcal{F})|$ and $M$ be the median of $h(\pi)$. We claim that $M$ is close to $p|\mathcal{F}|$.
    \begin{claim}\label{clm:median}
        $M = p |\mathcal{F}| \pm n^{3/4}$.
    \end{claim}
    \allowdisplaybreaks
    \begin{claimproof}[Proof of \Cref{clm:median}]
        For each element $e \in \mathcal{F}$, let $X_e$ be an indicator random variable for the event that $\sigma(\pi(e))=+1$.
        By the linearity of expectation, we have
        \[
        \EE[h] = \sum_{e \in \mathcal{F}} \EE[X_e] = p |\mathcal{F}|. 
        \]
        We now bound the variance of $h(\pi)$.
        Note that if $e \cap f = \emptyset$, then the covariance is 
        \begin{align*}
            \Cov(X_e,X_f) & =  \EE[X_e X_f] - \EE[X_e] \EE[X_f] \\
            & \leq \frac{pn^{(k)}(pn^{(k)}-1)}{n^{(2k)}} - p^2 \\
            & \leq p^2 \frac{n^k}{(n-2k)^k} - p^2 \\
            & \leq \frac{10k^2}{n}
        \end{align*}
        for sufficiently large $n$.
        For the case that $e \cap f \neq \emptyset$, as each element of $X$ is contained in at most $\Delta$ element of $\mathcal{F}$, there are at most $k\Delta |\mathcal{F}|$ pairs of elements $(e, f)$ of $\mathcal{F}$ that have nonempty intersection. 
        Therefore, 
        \begin{align*}
            \Var(h) &=  \sum_{e \in E(F)} \Var(X_e) + \sum_{e \neq f} \Cov(X_e, X_f) \\
            &\leq \EE[h] + \sum_{e \cap f \neq \emptyset} \Cov(X_e,X_f) +  \sum_{e \cap f = \emptyset} \Cov(X_e,X_f) \\
            & \leq p|\mathcal{F}| + k\Delta |\mathcal{F}| + \frac{10k^2}{n} |\mathcal{F}|^2 \\
             & \leq ((1+k\Delta)\Delta + 10k^2\Delta^2)n.
        \end{align*}
        Then by \Cref{lem:Chebyshev}, we have
        \[
            \PP(\left|h(\pi) - p|\mathcal{F}|\right| \geq n^{3/4}) \leq \frac{\Var(h)}{n^{3/2}} = O\left( \frac{1}{\sqrt{n}} \right).  
        \]
        Therefore, for sufficiently large $n$, the probability $\PP(|h(\pi) - p |\mathcal{F}|| \leq n^{3/4})$ is at least $1/2$, which implies that the median $M$ lies in the interval $[p|\mathcal{F}|-n^{3/4}, p|\mathcal{F}|+n^{3/4}]$.
    \end{claimproof}
    We now apply \Cref{lem:Talagrand}.
    If $\pi$ and $\pi'$ differ in at most $d$ coordinates, then there are at most $d\Delta$ elements $e$ of $\mathcal{F}$ satisfying $\pi(e) \neq \pi'(e)$. Thus, $|h(\pi) - h(\pi')| \leq 2d\Delta$.
    If $h(\pi) = s$, let $S$ be the set all $x\in X$ that are contained in one of the elements of $\pi(\mathcal{F}) \cap \sigma^{-1}(+1)$ and so $|S| \leq ks$. 
    If a permutation $\pi'$ agrees with $\pi$ on $S$, then $\pi'(\mathcal{F}) \cap \sigma^{-1}(+1)$ also contains all $s$ elements of $\pi(\mathcal{F}) \cap \sigma^{-1}(+1)$. Thus, $h(\pi') \geq s$.
    Therefore, by \Cref{lem:Talagrand}, we have
    \[
        \PP(|h(\pi) - p|\mathcal{F}| |\geq \varepsilon |\mathcal{F}|) \leq \PP(|h(\pi)-M| \geq \varepsilon^2 n/2) \leq 6\exp\left(-\frac{\varepsilon^4 n^2}{100k\Delta^2(M+\varepsilon^2 n)}\right) \leq \exp(-cn),
    \]
    which concludes the proof.
\end{proof}

By taking a union bound, we can simply deduce \Cref{thm:balanced_coloring}. Indeed, we can obtain a bit stronger result, which is the following.

\begin{theorem}\label{thm:balanced_coloring_stronger}
    Let $\Delta, \varepsilon, n_o, n$, $\{F_1, \dots, F_m\}$, and $\sigma$ are the same parameters, graph collection, and $\{-1, +1\}$ edge coloring as in the statement of \Cref{thm:balanced_coloring}. Assume initially $K_n$ has an edge decomposition $\{F_1, \dots, F_m\}$. Then the following holds.
    
    Let $\pi \in S_{V(K_n)}$ be a random permutation chosen uniformly at random. Then with probability at least $(1 - o(1))$, for all $i\in [m]$, we have $\mathrm{disc}^{\pm}(\pi(F_i)) = \mathrm{disc}^{\pm}(K_n) \pm \varepsilon$.
\end{theorem}

Note that \Cref{thm:balanced_coloring} is directly deduced from \Cref{thm:balanced_coloring_stronger}.

\begin{proof}[Proof of \Cref{thm:balanced_coloring_stronger}]
    Let $c=c(2\Delta, \varepsilon, 2)>0$ be a constant obtained from \Cref{lem:balanced_coloring}.
    For each $F_i$, let 
    $$\mathcal{F}_i = \left\{(u, v) \in V(K_n)^{(2)} \mid \{u, v\} \in E(F_i)\right\} \subseteq V(K_n)^{(2)}$$
    be the set of ordered edges.
    We also define $\sigma':V(K_n)^{(2)} \to \{-1, +1\}$ by letting $\sigma'((u, v)) = \sigma(\{u, v\})$.
    We observe that for any permutation $\pi \in S_{V(K_n)}$, we have $|\pi(\mathcal{F}_i) \cap {\sigma'}^{-1}(+1)| = 2|\pi(F_i) \cap \sigma^{-1}(+1)|$.
    
    Then by \Cref{lem:balanced_coloring}, for a uniform random permutation $\pi \in S_{V(K_n)}$, we have $\disc^{\pm}(\pi(F_i)) = \disc^{\pm}(K_n) \pm \varepsilon$ with probability at least $1-\exp(-cn)$ for each $i \in [m]$.
    As $m \leq \varepsilon^{-1} n$, by the union bound, with probability at least $1 - (\varepsilon^{-1}n \exp(-cn)) = 1 - o(1)$ such that $\disc^{\pm}(\pi(F_i)) = \disc^{\pm}(K_n) \pm \varepsilon$ for all $i \in [m]$.
\end{proof}


\subsection{Proof of Theorem~\ref{thm:high_discrepancy}}
If $\disc(K_{2n}) = \Omega(1)$, then by Theorem~\ref{thm:balanced_coloring}, we have a $1$-factor decomposition with high discrepancy so \Cref{thm:high_discrepancy} is clear. 
Hence, we may assume that $\disc(K_{2n}) = o(1)$.  
The key lemma of this section shows that under this condition, the number of $C_4$s in $K_{2n}$ with specific signing is large.
We call a cycle $C$ of length $4$ in $K_{2n}$ a \emph{switcher} if for its $1$-factor decomposition $\{\psi_1, \psi_2\}$, we have $S_{\sigma}(\psi_1) \neq S_{\sigma}(\psi_2)$.
There are six kinds of colored $C_4$ in colored $K_{2n}$ as illustrated in Figure~\ref{fig:all_C4}. The latter three types of colored $C_4$ are switchers while the first three are not. 
\begin{figure}[h]
    \centering
    \begin{subfigure}[t]{0.3\textwidth}
        \centering
        \begin{tikzpicture}
            \draw (0,0) rectangle (2,2);

            \node at (1,-0.3) {$+$};
            \node at (2.3,1) {$+$};
            \node at (1,2.3) {$+$};
            \node at (-0.3,1) {$+$}; 
        \end{tikzpicture}
        \caption{Type 1}
        \label{fig:tikz_a}
    \end{subfigure}
    \begin{subfigure}[t]{0.3\textwidth}
        \centering
        \begin{tikzpicture}
            \draw (0,0) rectangle (2,2);

            \node at (1,-0.3) {$-$};
            \node at (2.3,1) {$-$};
            \node at (1,2.3) {$-$};
            \node at (-0.3,1) {$-$}; 
        \end{tikzpicture}
        \caption{Type 2}
        \label{fig:tikz_b}
    \end{subfigure}
    \begin{subfigure}[t]{0.3\textwidth}
        \centering
        \begin{tikzpicture}
            \draw (0,0) rectangle (2,2);

            \node at (1,-0.3) {$+$};
            \node at (2.3,1) {$+$};
            \node at (1,2.3) {$-$};
            \node at (-0.3,1) {$-$}; 
        \end{tikzpicture}
        \caption{Type 3}
        \label{fig:tikz_c}
    \end{subfigure}
    \begin{subfigure}[t]{0.3\textwidth}
        \centering
        \begin{tikzpicture}
            \draw (0,0) rectangle (2,2);

            \node at (1,-0.3) {$+$};
            \node at (2.3,1) {$-$};
            \node at (1,2.3) {$+$};
            \node at (-0.3,1) {$-$}; 
        \end{tikzpicture}
        \caption{Type 4}
        \label{fig:tikz_d}
    \end{subfigure}
    \begin{subfigure}[t]{0.3\textwidth}
        \centering
        \begin{tikzpicture}
            \draw (0,0) rectangle (2,2);

            \node at (1,-0.3) {$+$};
            \node at (2.3,1) {$+$};
            \node at (1,2.3) {$+$};
            \node at (-0.3,1) {$-$}; 
        \end{tikzpicture}
        \caption{Type 5}
        \label{fig:tikz_e}
    \end{subfigure}
    \begin{subfigure}[t]{0.3\textwidth}
        \centering
        \begin{tikzpicture}
            \draw (0,0) rectangle (2,2);

            \node at (1,-0.3) {$-$};
            \node at (2.3,1) {$-$};
            \node at (1,2.3) {$-$};
            \node at (-0.3,1) {$+$}; 
        \end{tikzpicture}
        \caption{Type 6}
        \label{fig:tikz_f}
    \end{subfigure}
    \caption{Six kinds of colored $C_4$}
    \label{fig:all_C4}
\end{figure}
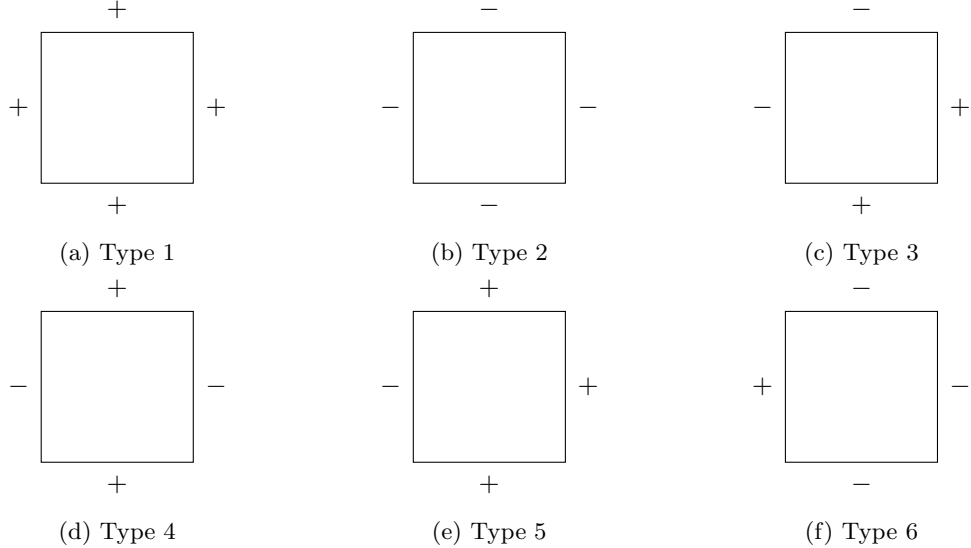

\begin{lemma}\label{Lem:switcher}
    Let $0<\delta<0.1$ be a constant. Then there exists a universal constant $\eta>0$ such that the following holds for every sufficiently large $n$.
    Let $\sigma:E(K_{2n}) \to \{-1, +1\}$ be a coloring of edges such that $\disc (K_{2m}) \leq \delta$.
    Then $K_{2n}$ contains at least $\eta n^4$ copies of switchers. 
\end{lemma}

To prove \Cref{Lem:switcher}, we apply the dependent random choice technique.

\begin{proof}[Proof of \Cref{Lem:switcher}]
    Let $N_1$ be the number of vertices $v$ such that $d^+(v) \geq 1.89n$ and let $N_2$ be the number of vertices $v$ such that $d^-(v) \geq 1.89n$.
    We claim that $N_1$ and $N_2$ cannot be close to $n$.
    \begin{claim}\label{clm:N_1N_2}
        $N_1 \leq 0.9n$ and $N_2 \leq 0.9n$.
    \end{claim}
    \begin{claimproof}[Proof of \Cref{clm:N_1N_2}]
        Suppose to the contrary that $N_1 > 0.9n$. 
        We enumerate the vertices of $K_{2n}$ in $v_1, \ldots, v_{2n}$ so that $d^+(v_1) \geq \cdots \geq d^+(v_{2n})$. Let $G^+$ be the graph formed by all positive edges in colored $K_{2n}$. Note that 
        
        \[
        \sum_{i=1}^{0.9n} d^+(v_i) \leq e(G^+[v_1, \ldots, v_{0.9n}]) + e(G^+) \leq \binom{0.9n}{2} + e(G^+).
        \]
        On the other hand, as we assume that $N_1 > 0.9n$, we have
        \[
        \sum_{i=1}^{0.9n} d^+(v_i) \geq 1.89n \times 0.9n = 1.701n^2.
        \]
        Combine the above two inequalities, we have
        \[
        1.296n^2 \leq 1.701n^2 - \binom{0.9n}{2} \leq e(G^+) \leq  \left(\frac{1}{2}+\frac{\delta}{2}\right) \binom{2n}{2},
        \]
        which leads to a contradiction.
        By the same argument, $N_2 \leq 0.9n$.
    \end{claimproof}
    We take a vertex $v \in V(K_{2n})$ uniformly at random. From now, let $c= \frac{1}{2000}$.
    \begin{claim}\label{clm:dpm}
        We have $\EE[d^+(v)d^{-}(v)] \geq cn^2$.
    \end{claim}
    \begin{claimproof}[Proof of \Cref{clm:dpm}]
        Note that,
        \begin{align*}
            \EE[d^+(v)d^{-}(v)] = \frac{1}{2n} \sum_{v \in V(K_{2n})} d^+(v)d^-(v) &= \frac{1}{2n} \sum_{v \in V(K_{2n})} d^+(v)(2n-1-d^+(v)) \\
            &\geq (2n-2) \sum_{v \in V(K_{2n})} t(v)(1-t(v)) ,
        \end{align*}
        where $t(v) = d^+(v)/(2n-1)$. By Claim~\ref{clm:N_1N_2}, at least $0.2n$ vertices satisfy $0.05 \leq t(v) \leq 0.95$, provided $n$ is sufficiently large.
        Thus, 
        \[
        \sum_{v \in V(K_{2n})} t(v)(1-t(v)) \geq 0.2n \cdot 0.05^2 = \frac{n}{2000}. 
        \]
        Therefore, 
        \[
        \EE[d^+(v)d^{-}(v)] = (2n-2) \sum_{v \in V(K_{2n})} t(v)(1-t(v)) \geq (2n-2) \cdot \frac{n}{2000} \geq \frac{n^2}{2000} = cn^2. 
        \]
    \end{claimproof}
    We define an ordered pair of vertices $(x,y) \in V(K_{2n})^2$ as \emph{good} if $|N^+(x) \cap N^-(y)| \geq \frac{n}{10^4}$, and as \emph{bad} otherwise. Let $Z$ be the number of bad pairs in $N^+(v) \times N^-(v)$.
    By linearity of expectation, 
    \[
    \EE[Z] = \sum_{\text{$(x, y)$ is a bad pair}} \PP\left((x, y) \in N^+(v) \times N^-(v) \right) \leq \sum_{\text{$(x, y)$ is a bad pair}} \frac{n}{10^4} \times \frac{1}{2n} \leq  \frac{n^2}{5000}. 
    \]
    Thus, we have
    \[
    \EE[d^+(v)d^{-}(v) - Z] \geq \frac{cn^2}{2}. 
    \]
    This implies that there exists a vertex $v \in V(K_{2n})$ for which there are at least $\frac{cn^2}{2}$ good pairs in $N^+(v) \times N^-(v)$. 

    Consider an auxiliary bipartite graph $G$ on vertex set $N^+(v) \cup N^-(v)$ as follows. Given vertices $x \in N^+(v), y\in N^-(v)$, $xy \in E(G)$ if and only if $(x,y)$ is a good pair. 
    Note that $d(G) \geq \frac{cn^2/2}{2n} = cn/4$, then there is a subgraph $H \subseteq G$ on vertex set $X \cup Y$ such that $\delta(H) \geq cn/8$, where $X \subseteq N^+(v), Y \subseteq N^-(v)$. The minimum degree condition gives $|X|,|Y| \geq cn/8$. 
    By averaging, we may assume there are at least $|X||Y|/2$ positive edges between $X$ and $Y$ in signed $K_{2n}$. 
    For each positive edge $xy$ with $x \in X$ and $y \in Y$, there are $cn/10$ choices of $z \in X$ such that $(z,y)$ is good. For each good pair $(z,y)$, there are $\frac{n}{10^4}$ choices of $w \in N^+(z) \cap N^-(y)$. Note that $xywz$ is a switcher for each choice of the above vertices. Hence, the number of switchers in $K_{2n}$ is at least
    \[
    \frac{1}{2} \cdot \frac{|X||Y|}{2} \cdot \frac{cn}{10} \cdot \frac{n}{10^4} \geq \frac{c^3}{2^8\cdot 10^5} n^4 = \frac{n^4}{2^{13}\cdot 10^{14}},
    \]
    where the leading multiplicity constant $\frac{1}{2}$ is due to each switcher being counted at most twice in the above enumeration. Hence, $\eta = \frac{1}{2^{13} \cdot 10^{14}} > 0$ is a desired universal constant. 
\end{proof}

We are now ready to prove \Cref{thm:high_discrepancy}.
\begin{proof}[Proof of Theorem~\ref{thm:high_discrepancy}]
    Let $\eta$ be a constant obtained in \Cref{Lem:switcher}. We note that one can choose $\eta = \frac{1}{2^{13} \cdot 10^{14}}$. Let $\gamma = \eta/(2 \times 10^4)$. 
    If $\mathrm{disc}(K_{2n}) > \gamma$, then by Theorem~\ref{thm:balanced_coloring}, we have a $1$-factor decomposition of $K_{2n}$ such that each matching has discrepancy at least $\gamma/2$. 
    We now assume that $\mathrm{disc}(K_{2n}) \leq \gamma$. 

    Let $F_1, \ldots, F_{n-1}$ be a decomposition of $K_{2n}$ obtained by Lemma~\ref{Lem:C4-decomp}. Note that independent from the parity of $n$, $F_i$ has at least $\frac{n}{2}-2$ copies of $C_4$ for $1 \leq i \leq n-2$ and $F_{n-1}$ contains at least $\frac{n}{2}-2$ copies of $K_4$.
    Let $\pi \in S_{V(K_{2n})}$ be a permutation chosen uniformly at random. 
    For each $i \in [n-2]$, let 
    \[
    \mathcal{F}_i = \left\{(v_1, v_2, v_3, v_4) \in V(K_{2n})^{(4)} \mid \text{$v_1 v_2 v_3 v_4$ is a cycle of length 4 in $F_i$} \right\} \subseteq V(K_{2n})^{(4)}.
    \]
    Define $\sigma': V(K_{2n})^{(4)} \to \{-1,+1\}$ by letting $\sigma'((v_1,v_2,v_3,v_4)) = +1$ if $v_1 v_2 v_3 v_4$ is a switcher in $K_{2n}$.
    By Lemma~\ref{Lem:switcher}, $K_{2n}$ has at least $\eta n^4$ switchers and so $|\sigma'(+1)| \geq 8\eta n^4$. 
    Thus, by applying Lemma~\ref{lem:balanced_coloring} and the union bound, we have 
    \[
    \left|\sigma'^{-1}(+1) \cap \pi(\mathcal{F}_i)\right| \geq \left( \frac{8\eta n^4}{(2n)^{(4)}} - \gamma \right)|\mathcal{F}_i| \geq \frac{\eta}{20}|\mathcal{F}_i|
    \]
    for each $i \in [n-2]$ with probability $1-o(1)$.
    Therefore, $\pi(F_i)$ contains at least $\frac{\eta n}{50}$ switchers for every $i \in [n-2]$.
    Similarly, we fix a $1$-factor decomposition $\psi_1, \psi_2, \psi_3$ of $F_{n-1}$. Then for each pair of indices $i \neq j \in [3]$, the union $\psi_i \cup \psi_j$ contains at least $\frac{n}{2}-2$ copies of $C_4$.
    Therefore, by applying the same argument, $\pi(\psi_i \cup \psi_j)$ contains at least $\frac{\eta n}{50}$ switchers for every $i \neq j \in [3]$ with probability $1-o(1)$.
    Furthermore, by \Cref{thm:balanced_coloring_stronger}, we have $\mathrm{disc}(\pi(F_i)) \leq 2\gamma$ for all $i \in [n-1]$ with probability $1-o(1)$.
            
    Therefore, by taking a union bound, there exists a permutation $\pi \in S_{V(K_{2n})}$ such that the following holds:
    \begin{itemize}
        \item $\mathrm{disc}(\pi(F_i)) \leq 2\gamma$ for every $i \in [n-1]$, 
        \item $\pi(F_i)$ contains at least $\eta n/50$ switchers for every $i \in [n-2]$, and 
        \item $F_{n-1}$ has a $1$-factor decomposition $\{\psi_1, \psi_2, \psi_3\}$ such that $\pi(\psi_i \cup \psi_j)$ contains at least $\eta n/50$ switchers for every $i \neq j \in [3]$.
    \end{itemize}
    By relabeling the vertices of $K_{2n}$, we may assume that $\pi$ is the identity permutation. We now claim that each $F_i$ can be decomposed into $1$-factors with high discrepancy.
    \begin{claim}\label{clm:1factor}
        For every $i \in [n-1]$, there is a $1$-factor decomposition of $F_i$ such that each 1-factor has a discrepancy at least $5\gamma$. 
    \end{claim}
    
    \begin{claimproof}[Proof of \Cref{clm:1factor}]
        If $i \in [n-2]$, then $F_i$ contains $\frac{\eta n}{50}$ switchers.
        Let $\{M_1, M_2\}$ be an arbitrary $1$-factor decomposition of $F_i$. 
        We may assume that there is a $k \in [2]$ such that $\disc(M_k) < 5 \gamma$, otherwise $M_1$ and $M_2$ is the desired 1-factor decomposition. Then $\disc(M_{2-k}) < 10\gamma$, since $\disc(F_i) \leq 2\gamma$. 

        Let $Q_1, \ldots, Q_m$ be switchers contained in $F_i$ where $m=\eta n/50$.
        Let $\psi^k_j := Q_j \cap M_k$ for $k \in [2]$ and let $J_1 \subseteq [m]$ be the set of indices $j \in [m]$ such that $S_{\sigma}(\psi^1_j) > S_{\sigma}(\psi^2_j)$ and $J_2 = [m] \setminus J_1$.
        Then either $|J_1| \geq \eta n/100$ or $|J_2| \geq \eta n/100$. 
        We only consider when $|J_1| \geq \eta n/100$ as the remaining case can be proved similarly.

        We replace $\psi^1_j$ in $M_1$ by $\psi^2_j$ for all $j \in J_1$ to get $M'_1$, and replace $\psi^2_j$ in $M_2$ by $\psi^1_j$ for all $j \in J_1$ to get $M'_2$.
        As $S_\sigma(\psi^1_j) -  S_\sigma(\psi^2_j)$ is either $2$ or $4$ for each $j \in J_1$ and $|J_1| \geq |S_{\sigma}(M_1)|$, we have 
        \[
        \disc(M'_1) = \left| \frac{S_\sigma(M'_1)}{n} \right| \geq \left| \frac{S_\sigma(M_1) - 2|J_1|}{n} \right| \geq \left| \disc(M_1)  - 2\frac{|J_1|}{n} \right| \geq \left| \frac{\eta}{50} - 5 \gamma\right| > 5\gamma,
        \]
        and
        \[
        \disc(M'_2) = \left| \frac{S_\sigma(M'_2)}{n} \right| \geq \left| \frac{S_\sigma(M_2) + 2|J_1|}{n} \right| \geq \left| \disc(M_2)  - 2\frac{|J_1|}{n} \right| \geq \left| \frac{\eta}{50} - 10\gamma \right|> 5\gamma.
        \]
        Therefore, $M'_1$ and $M'_2$ form a desired $1$-factor decomposition of $F_i$.

        We now consider the case when $i=n-1$.
        Let $\{\psi_1, \psi_2, \psi_3\}$ be a $1$-factor decomposition of $F_{n-1}$ chosen above.
        As each $\psi_i \cup \psi_j$ contains at least $\eta n/50$ switchers, we can select vertex disjoint switchers $Q_{1}^{12}, \ldots, Q_m^{12}, Q_{1}^{13}, \ldots, Q_m^{13}, Q_1^{23}, \ldots, Q_m^{23}$ with $m=\eta n/150$ such that for every $1 \leq t \leq m$ and $i < j \in [3]$, the switcher $Q_t^{ij}$ is contained in $\psi_i \cup \psi_j$.

        We now greedily make each $\psi_i$ have a large discrepancy.
        If $\disc(\psi_1) \geq 25 \gamma$, then we let $\psi_1' = \psi_1$ and $\psi_2' = \psi_2$.
        If $\disc(\psi_1) < 25 \gamma$, then let $J_1 \subseteq [m]$ be the set of indices $j \in [m]$ such that $S_{\sigma}(Q^{12}_j \cap \psi_1) > S_{\sigma}(Q^{12}_j \cap \psi_2)$.
        Let $J = J_1$ if  $|J_1| \geq m/2$ and $J=[m] \setminus J_1$ otherwise.
        We replace $Q^{12}_j \cap \psi_1$ in $\psi_1$ by $Q^{12}_j \cap \psi_2$ for all $j \in J$ to obtain $\psi_1'$ and replace $Q^{12}_j \cap \psi_2$ in $\psi_2$ by $Q^{12}_j \cap \psi_1$ for all $j \in J$ to obtain $\psi_2'$. 
        Then as $|J| \geq m/2 \geq |S_\sigma(\psi_1)|$, we have 
        $$\disc(\psi_1') \geq \left| \frac{S_\sigma(\psi_1) - 2|J_1|}{n} \right| \geq \left| \disc(M_1)  - 2\frac{|J_1|}{n} \right| \geq \left| \frac{\eta}{300} - 25 \gamma\right| > 25\gamma.$$
        
        If $\disc(\psi_2') \geq 5\gamma$, then let $\psi_2'' = \psi_2'$ and $\psi_3'=\psi_3$.
        If $\disc(\psi_2') < 5\gamma$, we carry on the same process. 
        Note that $Q^{12}_j$ and $Q^{23}_{j'}$ are vertex disjoint, so $Q^{23}_{j'}$ are contained in $\psi_2' \cup \psi_3$ for all $j' \in [m]$.
        Let $J_1' \subseteq [m]$ be the set of indices $j \in [m]$ such that $S_{\sigma}(Q^{23}_j \cap \psi_2') > S_{\sigma}(Q^{23}_j \cap \psi_3)$.
        Let $J' = J_1'$ if $|J_1'| \geq m/2$ and $J' = [m] \setminus J_1'$ otherwise.
        We construct $\psi_2''$ by replacing $Q^{23}_j \cap \psi_2'$ in $\psi_2'$ by $Q^{23}_j \cap \psi_3$ for all $j \in J'$ and $\psi_3'$ by replacing $Q^{23}_j \cap \psi_3$ in $\psi_3$ by $Q^{23}_j \cap \psi_2'$ for all $j \in J'$.
        Then by a similar computation, we have 
        $$\disc(\psi_2'') \geq |\frac{\eta}{300}-5\gamma| > 5\gamma.$$

        Finally, if $\disc(\psi_3') \geq 5\gamma$, then let $\psi_1''=\psi_1'$ and $\psi_3''=\psi_3'$.
        If $\disc(\psi_3') < 5\gamma$, then let $J_1''$ be the set of indices $j \in [m]$ such that $S_{\sigma}(Q^{13}_j \cap \psi_1') > S_{\sigma}(Q^{13}_j \cap \psi_3')$.
        If $|J_1''| \geq m/2$, then let $J'' \subseteq J_1''$ be a subset of arbitrary $5\gamma$ elements. If $|J_1''| < m/2$, then let $J'' \subseteq [m] \setminus J_1''$ be the set of arbitrary $5\gamma$ elements.
        We construct $\psi_1''$ by replacing $Q^{13}_j \cap \psi_1'$ in $\psi_1'$ by $Q^{13}_j \cap \psi_3'$ for all $j \in J''$ and $\psi_3''$ by replacing $Q^{13}_j \cap \psi_3'$ in $\psi_3'$ by $Q^{13}_j \cap \psi_1'$ for all $j \in J''$.
        As $2 \leq |S_\sigma(Q^{13}_j \cap \psi_1') - S_{\sigma}(Q^{13}_j \cap \psi_3')| \leq 4$ for each $j \in J'$, we have $$\disc(\psi_1'') \geq 25\gamma - 4\frac{|J|}{n} \geq 5\gamma,$$
        and
        $$\disc(\psi_3'') > 2\frac{|J|}{n} - 5\gamma \geq 5\gamma.$$
        Therefore, $\{\psi_1'', \psi_2'', \psi_3''\}$ is a desired decomposition.
    \end{claimproof}
    By the claim, each $F_i$ can be decomposed into $1$-factors with discrepancy at least $5\gamma$ and $\{F_1, \ldots, F_{n-1}\}$ is an edge decomposition of $K_{2n}$, so we obtain a desired $1$-factor decomposition of $K_{2n}$.
\end{proof}

\section{Concluding remark}
In this section, we would like to consider several generalizations of \Cref{thm:high_discrepancy}.
Indeed, \Cref{thm:high_discrepancy} can be generalized into the following ``unbalanced'' $1$-factor decomposition result. 
\begin{theorem}\label{thm:unbalanced}
    For every $p_0 \in (0, 1)$, there exist $c=c(p_0)>0$ and $n_0=n_0(p_0)$ such that the following holds for every $n \geq n_0$.
    Let $\sigma:E(K_{2n}) \to \{-1, +1\}$ be an edge coloring of $K_{2n}$ such that $\disc^{\pm}(K_{2n}) \in [-p_0, p_0]$.
    Then there exists a $1$-factor decomposition $\{\psi_1, \ldots, \psi_{2n-1}\}$ of $K_{2n}$ such that for all $i\in[m]$, we have $\left| \disc^{\pm}(\psi_i) - \disc^{\pm}(K_{2n}) \right| \geq c$.
\end{theorem}
The proof is the same as the proof of \Cref{thm:high_discrepancy} so we omit it here. 
As a simple corollary of this theorem, we can prove a multi-colored version of \Cref{thm:high_discrepancy}.
\begin{theorem}
    For every $k \geq 2$, there exists $c=c(k)>0$ such that the following holds for sufficiently large $n$.
    Let $\sigma:E(K_{2n}) \to [k]$ be an edge coloring of $K_{2n}$. Then there exists a $1$-factor decomposition $\mathcal{M}$ of $K_{2n}$ such that for each perfect matching $M \in \mathcal{M}$, there exists a color $i \in [k]$ that appears more than $\frac{n}{k}+cn$ times.
\end{theorem}
\begin{proof}[Proof sketch]
    By \Cref{thm:balanced_coloring}, we may assume that $\sigma$ is a balanced coloring. i.e., each color appears almost equal in time.
    Then we define new coloring $\sigma':E(K_{2n}) \to \{-1, +1\}$ by recoloring colors in $[k-1]$ by $+1$ and color $k$ by $-1$.
    Then by applying \Cref{thm:unbalanced}, we obtain a desired $1$-factor decomposition. 
\end{proof}

Another possible generalization of \Cref{thm:high_discrepancy} is the consideration of the Dirac-type problem. 
Csaba, K\"{u}hn, Lo, Osthus and Treglown~\cite{Dirac-1-factorization} proved that if $d \geq 2\lceil n/4 \rceil -1$, then every $d$-regular graph has $1$-factor decomposition, confirming the conjecture by Chetwynd and Hilton~\cite{Chetwynd-Hilton}. 
On the other hand, by adapting the construction of Balogh, Csaba, Jing, and Pluh\'{a}r~\cite{Discrepancy-Hamcycle}, there exists $(3n/4)$-regular graph such that all the perfect matchings have discrepancy $0$. Then it would be interesting to consider the threshold of $1$-factor decomposition with high discrepancy.
\begin{question}
    What is the smallest $\delta>0$ such that the following holds:
    If $G$ is $d$-regular with $d \geq (\delta + o(1))n$, then for every edge coloring $\sigma:E(G) \to \{-1, +1\}$, there exists a $1$-factor decomposition of $G$ such that each perfect matching in the decomposition has discrepancy $\Omega(1)$. 
\end{question}

We finally want to consider Hamilton cycle decompositions. It is well-known~\cite{HILTON1984125} that $K_{2n+1}$ has a decomposition into Hamilton cycles.
The result by Balogh, Csaba, Jing and Pluh\'{a}r~\cite{Discrepancy-Hamcycle}, for every edge coloring $\sigma:E(K_{2n+1}) \to \{-1, +1\}$, one can find at least $(1/8-o(1))n$ edge-disjoint Hamilton cycles with discrepancy $\Omega(1)$.
We would like to ask whether it is possible to decompose the edge set of $K_{2n+1}$ into Hamilton cycles such that each has discrepancy $\Omega(1)$. 


\subsection*{Acknowledgement}
JA is supported by the National Natural Science Foundation of China under grant No.12161141006 and No.12401456.
SI and HL are supported by the National Research Foundation of Korea (NRF) grant funded by the Korea government(MSIT) No. RS-2023-00210430, and supported by the Institute for Basic Science (IBS-R029-C4).

This research was performed during the third and the fourth authors' visits to Nankai University. They thank Nankai University for their hospitality and for providing a great working environment.


\printbibliography

\end{document}